\documentclass[a4paper,10pt]{amsart}

\usepackage[utf8]{inputenc}
\usepackage{amsmath, amsthm, amssymb, comment}
\usepackage[dvips]{graphicx}
\usepackage{overpic}
\usepackage{enumitem}
\usepackage[colorlinks=true, citecolor=cyan]{hyperref}
\usepackage{pinlabel}

\usepackage[margin=2.5cm]{geometry}

\usepackage{color}

\newtheorem{theorem}{Theorem}[section]
\newtheorem{corollary}[theorem]{Corollary}
\newtheorem{proposition}[theorem]{Proposition}
\newtheorem{definition}[theorem]{Definition}
\newtheorem{lemma}[theorem]{Lemma}

\newtheorem{conjecture}[theorem]{Conjecture}
\newtheorem*{theorem*}{Theorem}
\newtheorem*{proposition*}{Proposition}
\newtheorem*{definition*}{Definition}
\newtheorem*{lemma*}{Lemma}
\newtheorem*{claim*}{Claim}
\newtheorem*{corollary*}{Corollary}
\newtheorem*{convention*}{Convention}

\theoremstyle{definition}
\newtheorem{convention}[theorem]{Convention}

\newtheorem{question}{Question}

\theoremstyle{remark}
\newtheorem{rem}[theorem]{Remark}
\newtheorem*{rem*}{Remark}

\newcommand{\wt}[1]{\widetilde{#1}}

\newcommand\bR{\mathbb R}

\newcommand\bH{\mathbb H}

\newcommand{\R}{\mathbb R}

\newcommand\eps{\varepsilon}

\newcommand\orb{ \mathcal O }

\newcommand\fs{\mathcal F^{s} }

\newcommand\fu{\mathcal F^{u} }

\newcommand{\cF}{\mathcal{F}}

\newcommand{\cP}{\mathcal{P}}

\newcommand{\cO}{\mathcal{O}}

\newcommand{\Pmod}{\mathrm{PMod}}

\newcounter{notes}
\title{A Sm\"org\aa sbord of (bi)contact structures, Reeb flows and pseudo-Anosov flows}

\author[Thomas Barthelm\'e]{Thomas Barthelm\'e}
\address{Queen's University, Kingston, Ontario}
\email{thomas.barthelme@queensu.ca}
\urladdr{sites.google.com/site/thomasbarthelme}

\begin{document}

 \begin{abstract}
 This note is meant as an invitation to the study of certain families of contact structures, centering around the following question: ``How much can one relate the dynamics of two distinct Reeb flows of the same contact structure?''
We gather some results as well as state many more questions and conjectures around that theme.
 \end{abstract}

 \maketitle
 
\section{Introduction}

A given (coorientable) contact structure $\xi$ on a $3$-manifold $M$ admits many Reeb flows. From a dynamical point of view, different Reeb flows of the same contact structure may, a priori, behave differently. However, in some cases, dynamical similarities have started to appear. There are for instance many examples of contact structures for which every Reeb flows has positive topological entropy (see, e.g., \cite{Alv16,CDHR24,CDR23,FHV}). Positivity of topological entropy is not a strong dynamical relationship between the Reeb flows, as it only implies that each flow exhibits some type of hyperbolic behavior. However, several of these results give more, as they in fact show some correspondence between the conjugacy classes of loops that admit a periodic Reeb orbit.
In some cases, one can in fact hope for a very strong dynamical relationship. In \cite{BMB}, J.~Bowden, K.~Mann, and I proved the following result:
\begin{theorem}[\cite{BMB}]\label{thm_BMB}
Let $R_1$, $R_2$ be two Anosov Reeb flows of two contact structures $\xi_1,\xi_2$. Then the Anosov flows $R_1, R_2$ are orbit equivalent if and only if $\xi_1$ and $\xi_2$ are contactomorphic.
\end{theorem}

In this note, I want to invite the reader to think about how much more general such a result could possibly be made. The most general question I would like to ask is 
\begin{question}\label{q_unprecise}
Given two Reeb flows $R_1$, $R_2$, of the same contact structure $\xi$. How ``related'' are the dynamics of $R_1$, $R_2$?
\end{question}

Here, ``related'' could take many different meanings from the strongest (e.g., orbit equivalent, as in the above theorem), to some much weaker form (e.g., both have positive topological entropy, as discussed previously).
In this note, I will consider two different notions of what ``related'' should mean. The first, and a priori weakest, involves the \emph{free homotopy data}, defined as follows
\begin{definition}\label{def_free_hom_data}
Given a flow $\phi$ on $M$, its \emph{free homotopy data}, denoted by $\cP(\phi)$ is the set of free homotopy classes represented by unoriented periodic orbits of $\phi$. That is,
\[
\cP(\phi) :=\{ [\alpha] \in [\pi_1(M)] \mid [\alpha] \text{ or } [\alpha^{-1}] \text{ is represented by a periodic orbit of } \phi\},
\]
where $[\alpha]$ denotes the conjugacy class of an element of $\pi_1(M)$.
\end{definition}
The free homotopy data was shown in \cite{BMB,BFM} to be a complete invariant for a lot of pseudo-Anosov flows (see Theorem \ref{thm_free_homotopy_data} below for a precise statement). In my obviously very biased opinion, this result indicates that studying the free homotopy data of more general classes of $3$-dimensional flows should be interesting.
So, a rephrasing of Question \ref{q_unprecise} using this notion is:
\begin{question}\label{que_reeb_same_free_homotopy_data}
Given two Reeb flows $R_1$, $R_2$, of the same contact structure $\xi$.
Under what conditions on the Reeb flows $R_i$, do we have that $\cP(R_1)= \cP(R_2)$?
Does there exist a Reeb flow $R_{\mathrm{min}}$ of $\xi$ such that for every other Reeb flow $R$ of $\xi$, $\cP(R_{\mathrm{min}})\subset \cP(R)$
\end{question}

Deep works in contact geometry, that we will review in section \ref{sec_free_homotopy_data}, actually give some conditions answering Question \ref{que_reeb_same_free_homotopy_data}, at least when $\xi$ is left invariant by some Anosov flow, and this was crucially used in the proof of Theorem \ref{thm_BMB}. 
Note that a recent very interesting work of Jonathan Zung \cite{Zung} gives conditions for certain flows to share the same free homotopy data using \emph{stable Hamiltonian structures}, instead of contact ones, and obtains a generalization of Theorem \ref{thm_BMB}.

Another notion of what ``related'' can mean that I would like to advertise, is that of two flows having the same \emph{pseudo-Anosov model}, in the following sense:
\begin{definition}\label{def_pAmodel}
Let $\phi$ be any flow on $M$, and $\psi$ a pseudo-Anosov flow on $M$. We say that $\psi$ is \emph{a pseudo-Anosov model} of $\phi$ if there exists $K$ a $\phi$-invariant compact subset of $M$ and a continuous and surjective map $h \colon K \to M$ that takes flowlines of $ \phi|_K$ to flowlines of $\psi$
\end{definition}
For surface homeomorphisms, Handel \cite{Han85} showed that if a homeomorphism $g$ is homotopic to a pseudo-Anosov map $f$, then the suspension of $f$ is a pseudo-Anosov model of the suspension of $g$ as in the definition above. Thus, I like to think of pseudo-Anosov models and flows admitting pseudo-Anosov models, as an interesting class in a hypothetical effort to try to generalize the Nielsen--Thurston Classification to, at least some, $3$-manifold flows.

With this definition, the other rephrasing of Question \ref{q_unprecise} that I would like to advertise is:
\begin{question}\label{q_general_samepAmodel}
Given two Reeb flows $R_1$, $R_2$, of the same contact structure $\xi$.
Under what conditions on the contact structure $\xi$, and/or the Reeb flows $R_i$, do we have that $R_1$ and $R_2$ admit the same pseudo-Anosov model?
\end{question}

Notice that an Anosov flow is by definition its own pseudo-Anosov model, and thus Theorem \ref{thm_BMB} gives a condition for a positive answer to Question \ref{q_general_samepAmodel}.
Obviously, one would like to also decide when a given Reeb flow even admits a pseudo-Anosov model:
\begin{question}\label{que_existencepAmodel}
Given a Reeb flow $R$ of a contact structure $\xi$. 
Under what conditions on the contact structure $\xi$, and/or the Reeb flow $R$, can we ensure that $R$ admits a pseudo-Anosov model?
\end{question}

While the existence of a pseudo-Anosov model would give a lot of information about the dynamics of the Reeb flow, it does not give \emph{all} the information. In particular, it is not clear that a pseudo-Anosov model is necessarily unique (up to orbit equivalence). Knowing which Reeb flows admit such a unique pseudo-Anosov model would be very interesting, as it would indicate that the pseudo-Anosov model captures in some sense all the dynamical information that is somewhat ``robust'', i.e., cannot be destroyed by small perturbations.

\begin{question}\label{q_unicity_pA_model}
Under what conditions on the flow do we have uniqueness of the pseudo-Anosov model?
\end{question}

In the rest of this note, I will recall the main step of the proof of Theorem \ref{thm_BMB} (in section \ref{sec_free_homotopy_data}) and consider subclasses of contact structures and Reeb flows that may be good candidates for studying Question \ref{q_general_samepAmodel}. In particular, I will focus in section \ref{sec_anosov_supporting} on \emph{bicontact} structures that supports Anosov flows (see Definition \ref{def_Anosov_supporting_contact}). That section also contains a few results that are not available in the existing literature about the relationships between Anosov-supporting contact structures and those supporting an Anosov Reeb flow (see Theorem \ref{thm_bitransverse_Anosov_imply_skew} and Proposition \ref{prop_Anosov_contact_implies_Anosov_supporting}).
In section \ref{sec_pAmodels}, I discuss the existence of pseudo-Anosov models.

\subsection*{Acknowledgement}

This note started as a talk I gave in the workshop ``Symplectic Geometry and Anosov Flows'' held in Heidelberg in July 2024. The aim of the talk, and subsequently of this note, was to advertise some questions at the intersection of contact geometry and hyperbolic dynamics that I find interesting, and hoped other people may be persuaded to care about. Hence this note contains many more questions and conjectures than answers. I also allowed myself to be slightly less formal than in my other articles.

I would like to thank Jonathan Bowden, Surena Hozoori, Katie Mann, Federico Salmoiraghi and Jonathan Zung for a lot of discussions related to some of the questions, or results, described below. 
I would also like to thank all the organizers, the speakers, and the attendees of that workshop, as I had fun and learned a lot. In particular, I finally know that stable is blue and unstable is red. 

Finally, I am partially supported by the NSERC (RGPIN-2024-04412), so thanks Canada.

\section{Contact Anosov structures and $\bR$-covered flows} 

We quickly recall some definitions, much more details can be found for instance in Rafael Potrie's lecture notes \cite{Potrie_notes} in the present volume.

To ease notations, we will generally write $X,Y,R,\dots$ for both vector fields and their associated flows, and write $X^t,Y^t,R^t,\dots$ when the time-parameter of the flow is actually needed.

\begin{convention}
We will always assume that our vector fields are nonsingular and at least $C^1$.
\end{convention}

 \begin{definition}
 Let $M$ be a closed manifold equipped with a Riemannian metric.
 A vector field $X$ on $M$ is said to be (smooth) \emph{Anosov} if there exists a splitting of the tangent bundle ${TM =  \R\cdot X \oplus E^{ss} \oplus E^{uu}}$ preserved by $DX^t$ and two constants $a,b >0$ such that:
\begin{enumerate}[label = (\roman*)]
 \item For any $v\in E^{ss}$ and $t>0$,
    \begin{equation*}
     \lVert DX^t(v)\rVert \leq be^{-at}\lVert v \rVert \, ;
    \end{equation*}
  \item For any $v\in E^{uu}$ and $t>0$,
    \begin{equation*}
     \lVert DX^{-t}(v)\rVert \leq be^{-at}\lVert v \rVert\, .
    \end{equation*}
\end{enumerate}
\end{definition}
If $b=1$, we say the metric on $M$ is {\em adapted} to the flow (such a metric is also sometimes called a \emph{Lyapunov metric}).  Using an ``averaging trick" one can show that any manifold with a smooth Anosov flow $X^t$ admits an adapted metric for $X^t$.  See \cite[Prop 5.1.5]{FH_book} for a proof. 

The classical \emph{Stable Manifold Theorem} shows that an Anosov flow defines invariant foliations:
\begin{proposition}[Anosov \cite{Ano67}] \label{prop_integrable}
 If $X^t$ is a smooth Anosov flow, then the distributions $E^{ss}$, $E^{uu}$, $E^{ss} \oplus \R \cdot X$, and $E^{uu} \oplus \R \cdot X$ are uniquely integrable. The associated foliations are denoted respectively by $\cF^{ss}$, $\cF^{uu}$, $\fs$, and $\fu$ and called the strong stable, strong unstable, (weak) stable, and (weak) unstable foliations, respectively.
\end{proposition}

An Anosov flow is called \emph{transversally orientable} if \emph{both} stable and unstable foliations are transversally orientable, or equivalently, if both $E^{ss}$ and $E^{uu}$ are orientable line fields. In particular, since $X$ is a non-singular vector field, it implies that $M$ is orientable.
\begin{convention}
In this text, we will always assume that the manifolds we consider are orientable and the Anosov flow we consider are transversally orientable.
\end{convention}

More generally, we recall the definition of a (topological) pseudo-Anosov flow:
\begin{definition}\label{def_topPA}
A (topological) \emph{pseudo-Anosov} flow on a manifold $M$ is a flow $X^t$ generated by a continuous, non-singular, vector field $X$, 
satisfying the following conditions.  
\begin{enumerate}[label = (\roman*)]
\item There are two topologically transverse, $2$-dimensional, singular foliations $\cF^{s}$ and $\cF^{u}$, invariant under $X^t$, with singularities along a finite (possibly empty) collection of orbits $\alpha_1, \dots, \alpha_n$. 
\item \label{item_PAF_prongs}
Each singular orbit $\alpha_i$ has a neighborhood homeomorphic via a foliation-respecting homeomorphism, to that of a model $p_i$-prong, for some $p_i \geq 3$. (See Figure \ref{fig_3_prong} or, e.g., \cite[Definition 5.8]{AgolTsang} for a definition of $p$-prong).
\item  \label{item_converges} 
Given any $x\in M$ and $y \in \cF^{s}(x)$ (resp.~$y \in \cF^{u}(x)$), there exists a continuous increasing reparameterization $h\colon \R \to \R$ such that $d\left(X^t(x), X^{h(t)}(y)\right)$ converges to $0$ as $t\to +\infty$ (resp.~$t\to -\infty$). 
\item \label{item_PAF_backwards_expansivity} 
There exists $\eps>0$ such that for any $x\in M$ and any $y\in \cF^{s}_{\eps}(x)$ (resp.~$y\in \cF^{u}_{\eps}(x)$), with $y$ not on the same orbit as $x$, then for any continuous increasing reparameterization $h\colon \R \to \R$, there exists a $t\leq 0$ (resp.~$t\geq 0$) such that $d\left(X^t(x), X^{h(t)}(y)\right) >\eps$. 
\end{enumerate}
\end{definition}

\begin{figure}[h]
\includegraphics[height=5cm]{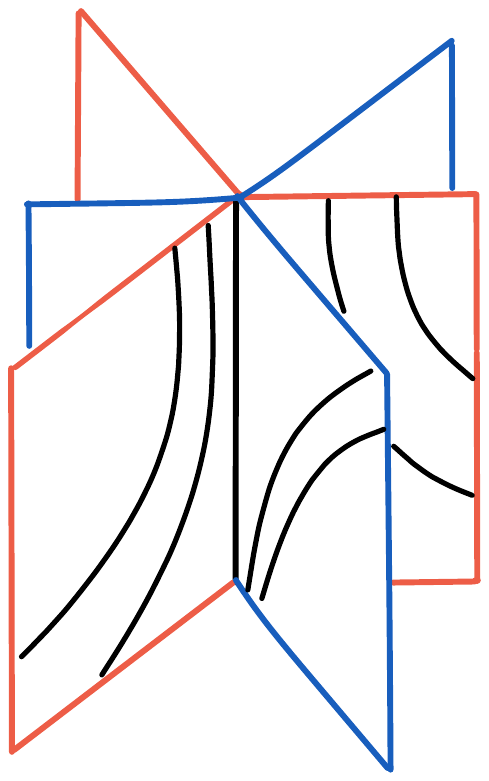}
\caption{The local picture of stable and unstable foliations near a $3$-prong singularity}
\label{fig_3_prong}
\end{figure}

One can also give a \emph{smooth} version of a pseudo-Anosov flow, where one assume that away from the singular orbits, the flow is smooth Anosov, and that the homeomorphisms of item \ref{item_PAF_prongs} can be made Lipshitz, see \cite[Definition 5.9]{AgolTsang}.

In section \ref{sec_pAmodels}, we will encounter \emph{$1$-pronged pseudo-Anosov flow}. Those are pseudo-Anosov flows as in the above definition but for which some of the singular orbits may be $1$-prongs.

Recall that a Reeb vector field $R$ associated to a contact form $\alpha$ is the unique vector field such that $\alpha(R) = 1$ and $i_Xd\alpha =0$. 

\begin{definition}
We say that a contact structure $\xi$ is \emph{Anosov} if $\xi$ admits a Reeb vector field $R$ that is Anosov. In that case, we say that $R$ is Reeb-Anosov.
\end{definition}

\subsection{More background}

We recall here some more facts about pseudo-Anosov flows. We again refer to the lecture notes \cite{Potrie_notes} in this volume for more details and references.

\begin{definition}
We say that two flows $X_1,X_2$ on $M$ are \emph{orbit-equivalent} if there exists a homeomorphism $h\colon M \to M$ sending (unoriented) orbits of $X_1$ to (unoriented) orbits of $X_2$\footnote{The classical definition in dynamical systems is to ask for an orbit equivalence to preserve the direction of the orbits, but from our point of view it is often easier to allow the orbit equivalence to reverse all orientations.}.

If $h$ is homotopic to identity, we say that $X_1,X_2$ are \emph{isotopically equivalent}.
\end{definition}

Given an Anosov flow $X^t$ on $M$, we write $\wt M$ for the universal cover of $M$ and $\wt X^t$ for the lifted flow.
Let $\cO_X=  \wt M / \wt X^t$ denote the \emph{orbit space} of $X$, recall that a fundamental result due to Barbot \cite{Bar_caracterisation} and Fenley \cite{Fen_Anosov_flow_3_manifolds} (and Fenley--Mosher \cite{FenMosher} for the pseudo-Anosov case) is that $\orb_X \simeq \bR^2$.
Moreover, the lifted foliations $\wt \cF^{s,u}$ of $X$ descend to $1$-dimensional foliations of $\cO_X$, so that $(\cO_X,\wt \cF^{s},\wt \cF^{u})$ is a bifoliated plane, and the action of $\pi_1(M)$ on $\wt M$ induces an action on that bifoliated plane. 

Among all the bifoliated planes that arise as orbit spaces of Anosov flows, two are particularly simple: the trivial and skew planes.

\begin{definition}
A bifoliated plane $(P,\cF^+,\cF^-)$ is called:
\begin{itemize}
\item \emph{trivial} if it is isomorphic\footnote{i.e., homeomorphic via a foliation preserving homeomorphism.} to $\bR^2$ equipped with the foliations by vertical and horizontal lines;
\item \emph{skew} if it is isomorphic to the diagonal strip $\{(x,y)\in \bR^2 \mid x<y<x+1\}$ equipped with the foliations by vertical and horizontal lines.
\end{itemize}
\end{definition}

For transversely orientable Anosov flows with skew orbit space, one can further record whether the isomorphism to the diagonal strip $\{(x,y)\in \bR^2 \mid x<y<x+1\}$ preserves the transverse orientations or not.
\begin{definition}
A transversely orientable Anosov flow is called \emph{positively skew} if its orbit space is isomorphic to $\{(x,y)\in \bR^2 \mid x<y<x+1\}$ via an homeomorphism preserving the transverse directions.
It is called \emph{negatively skew} otherwise.
\end{definition}

Just from the definitions, one sees that an isotopic equivalence between pseudo-Anosov flows induces an isomorphism between their orbit spaces that conjugates the respective actions of $\pi_1(M)$. Another essential result of Barbot \cite{Bar_caracterisation} was that conversely, a conjugacy between the induced $\pi_1(M)$-actions on orbit spaces induces an isotopic equivalence between the pseudo-Anosov flows. 

Similarly to the orbit space, one can also define the \emph{(un)stable leaf space} $\Lambda(\cF^s)$ ($\Lambda(\cF^u)$) as the quotient of $\wt M$ by the lifted (un)stable foliations of $X$. In general, the leaf spaces of Anosov flows are $1$-manifolds, and that of pseudo-Anosov flows are \emph{non-Hausdorff trees} (see, e.g., \cite{Fen03}). It turns out that when one of these leaf space happens to be Hausdorff, then it forces the orbit space to be one of the special case we defined earlier. This is the trichotomy result of Barbot and Fenley for pseudo-Anosov flows:
\begin{theorem}[Barbot, Fenley]
Let $X$ be a pseudo-Anosov flow on $M$, then exactly one of the following possibilities happen:
\begin{enumerate}[label=(\roman*)]
\item The orbit space $\cO$ is trivial, which happens if and only if $X$ is the suspension of an Anosov diffeomorphism;
\item The orbit space is skew;
\item $X$ has at least one singular orbit, or both leaf spaces $\Lambda(\wt \cF^s)$ and $\Lambda(\wt \cF^u)$ are not Hausdorff.
\end{enumerate}
\end{theorem}

In the first two cases of the trichotomy above, the leaf spaces $\Lambda(\wt \cF^s)$ and $\Lambda(\wt \cF^u)$ are both homeomorphic to $\bR$, and we call the flow $\bR$-covered. 

Barbot \cite{Bar01} proved that Reeb-Anosov flows are skew (see also \cite{Mar_easy} for a shorter argument). In \cite{Mar24}, Marty proved the following remarkable result:
\begin{theorem}[Marty \cite{Mar24}]\label{thm_martyequivalence}
An Anosov flow is positively (resp.~negatively) skew if and only if it is isotopically equivalent to a Reeb-Anosov flow of a positive (resp.~negative) contact structure.
\end{theorem}

We finally recall two more notions that will appear later:
A taut foliation $\cF$ of $M$ is called \emph{$\bR$-covered} if its leaf space $\Lambda(\cF)$ is homeomorphic to $\bR$. The leaf space $\Lambda(\cF)$ is defined as above as the quotient of $\wt M$ by the lifted foliation $\wt \cF$.

Given an $\bR$-covered foliation, one can build transversals in $\wt M$ that intersects \emph{every} leaf of the lifted foliations $\wt \cF$. This is particularly interesting if these transversals can be chosen to be the lifted orbits of a flow on $M$. Thus we define:
\begin{definition}\label{def_regulating}
A vector field $X$ transverse to a foliation $\cF$ is called \emph{regulating} for $\cF$ if in the universal cover every orbit of $\wt X$ intersects every leaf of $\wt \cF$.
\end{definition}
Notice that if a foliation admits a regulating flow then it it automatically $\bR$-covered\footnote{In fact it is even more: it is \emph{uniform}, see \cite{Calegari_book} for more on the distinctions between uniform and $\bR$-covered foliations.}.

\section{The free homotopy data and orbit equivalences}\label{sec_free_homotopy_data}

In this section, we recall the main steps of the proof of Theorem \ref{thm_BMB}, focusing on the implication that the isotopy class of an Anosov contact structure uniquely determines the Reeb-Anosov flows, up to isotopy equivalence.


The proof of Theorem \ref{thm_BMB} has two main steps:
\begin{enumerate}
\item Show that the \emph{free homotopy data} (Definition \ref{def_free_hom_data}) of a Reeb-Anosov flow uniquely determines the flow;
\item Use contact magic\footnote{Maybe this does not feel like magic to any self-respecting contact geometer, but I am not one of them.}: The cylindrical contact homology allows one to show that the free homotopy data of \emph{any} Reeb-Anosov flow of a fixed contact structure are the same.
\end{enumerate}

The first step is given by the following theorem:
\begin{theorem}[\cite{BMB,BFM}]\label{thm_free_homotopy_data}
Let $X_1$, $X_2$ be Anosov flows on $M$. Assume that $X_1$ is $\bR$-covered, or that there are no tori transverse to $X_1$.
Then $\cP(X_1)=\cP(X_2)$ if and only if $X_1$ and $X_2$ are isotopically equivalent.
\end{theorem}
The ``no transverse tori'' assumption, is satisfied by a wide range of pseudo-Anosov flows, for instance whenever $M$ is hyperbolic (see Lemma \ref{lem_transversetorus} below). In particular skew flows do not admit transverse tori was proved by Barbot \cite[Théorème C]{Bar_caracterisation}. 

The result stated above is actually a special case of the main result of \cite{BFM} which gives a complete invariant of isotopic equivalence for transitive pseudo-Anosov flows. Since that special case is not stated as such in \cite{BMB,BFM}, we start by showing that the statement we give above do follow from \cite{BFM}. To do that, we need the following lemma 
\begin{lemma}\label{lem_transversetorus}
Let $T$ be a torus transverse to a pseudo-Anosov flow $X$. Then $T$ is incompressible in $M$. 
\end{lemma}
This lemma was proved by Fenley \cite[Corollary 2.2]{Fenley_QGAF} for Anosov flows, see also \cite[Lemma 1]{Brunella} and \cite[Proof of Proposition 2.7]{Mosher_homologynormI} for a special version of this lemma. The proof below was told to me by Kathryn Mann.
\begin{proof}[Proof sketch]
Suppose $T$ is compressible. Let $\wt T$ be a lift of $T$ to $\wt M$. Then $\wt T$ must have non-trivial $\pi_1$, and, since $T$ is transverse to $X$, any orbit of $\wt X$ intersects $\wt T$ at most once. In particular, $\wt T$ cannot be a torus (otherwise there would be recurrent orbits of $\wt X$, a contradiction), so must be an annulus. In particular, there exists $g\in \pi_1(M)$, non-trivial, that leaves $\wt T$ invariant.
Consider a fundamental domain of $\wt T$ for the action of $g$, and call $A$ its projection to the orbit space $\cO_X$. Since orbits of $\wt X$ intersects $\wt T$ at most once, $A$ must be an annulus, and thus have an ``inner'' boundary $c$ (i.e., $c$ is the boundary of $A$ that bounds a topological disc disjoint from $A$). Since $A$ is a fundamental domain for $g$, $gc$ must be the outer boundary of $A$, but this is impossible for the action of $g$ on $\cO$: $g$ would need to have at least one fixed point in the disc bounded by $c$, and since all fixed points of $g$ are hyperbolic, one must have $gc\cap c \neq \emptyset$.
\end{proof}

We may now justify the statement of Theorem \ref{thm_free_homotopy_data}:
\begin{proof}[Proof of Theorem \ref{thm_free_homotopy_data}]
If $X_1$ is $\bR$-covered, then the statement is Theorem 1.2 in \cite{BMB}.
If $X_1$ does not admit any transverse torus, it must be transitive (see Brunella \cite{Brunella} for the Anosov case and \cite{BBM24b} for the general case). Hence, the statement of Theorem \ref{thm_free_homotopy_data} follows from \cite[Proposition 2.1]{BFM}, thanks to Lemma \ref{lem_transversetorus}.
\end{proof}

Now we can describe the contact magic.

First recall that a periodic orbit of a Reeb flow is called \emph{nondegenerate} if the linearized Poincaré map on that orbit has no roots of unity as eigenvalues. In dimension $3$, nondegenerate periodic orbits are either elliptic (if the eigenvalues are on the unit circle, but not roots of unity) or hyperbolic.
 A Reeb flow $R$ and its defining contact $1$-form are called nondegenerate if all the periodic orbits of $R$ are nondegenerate.
Notice that if $R$ is Reeb-Anosov, then every periodic orbit is hyperbolic, and thus $R$ is nondegenerate.

Cylindrical contact homology was introduced, in a more general context by Eliashberg--Givental--Hofer in \cite{EGH}. It was proved to be well-defined by Hutchings and Nelson \cite{HN} for \emph{dynamically convex} contact forms. Dynamically convex contact $1$-forms include all nondegenerate $1$-forms such that their Reeb flow have no contractible periodic orbits. In particular, it contains the class of contact $1$-forms of Reeb-Anosov flows.

 Some of the fundamental results in this theory can be summarized as follows:
\begin{theorem} 
 If $\alpha_1$ and $\alpha_2$ are nondegenerate dynamically convex contact forms on a closed, hypertight\footnote{A contact structure is called hypertight if it admits a Reeb flow with no contractible orbit.}, contact 3-manifold $(M ,\xi)$, then $C\bH^{\Lambda}_{\text{cyl}}(\alpha_1) \cong C\bH^{\Lambda}_{\text{cyl}}(\alpha_2)$ for any set $\Lambda$ of free homotopy classes in $M$.  The space $C\bH^{\Lambda}_{\text{cyl}}(\alpha_i)$ is a $\mathbb{Q}$-vector space, it is the homology of a complex generated by the periodic orbits of the Reeb flow of $\alpha_i$ in the free homotopy classes belonging to $\Lambda$. 
 \end{theorem}

The consequence of this result that is particularly relevant is that it \emph{almost} gives that the free homotopy data of any nondegenerate Reeb flow of a given hypertight contact structure $\xi$ are the same. Indeed, the above implies:
\begin{corollary}
Let $R_1,R_2$ be two nondegenerate Reeb flows of dynamically convex contact $1$-forms $\alpha_1,\alpha_2$ defining the same hypertight contact structure. Suppose $[g]\in \cP(R_1)$ is such that $C\bH^{[g]}_{\text{cyl}}(\alpha_1)\neq \{0\}$, then $[g]\in \cP(R_2)$.
\end{corollary}

In order to have a precise control of the free homotopy class, one then need to show that the associated cylindrical homology does not vanish. 


This is what Macarini and Paternain \cite[Theorem 2.1]{MP}, and independently Vaugon \cite{Vaug_unpublished} (see also \cite[Section 5]{FHV}) proved in the case of Reeb-Anosov flows. Hence we have:
\begin{proposition}[Macarini--Paternain, Vaugon]\label{prop_Reeb_flows_same_free_hom_data}
If $R_1$, $R_2$ are two Reeb-Anosov flows of the same Anosov contact structure $\xi$, then $\cP(R_1) = \cP(R_2)$. Moreover, for any other nondegenerate Reeb flow $R$ of $\xi$, $\cP(R_i) \subset \cP(R)$.
\end{proposition}

The key remark behind the proof of the above statement is that, for contact cylindrical homology, the differential of an odd (resp.~even) periodic orbit contains only even (resp.~odd) periodic orbits, where a non-degenerate periodic orbit is called even if it is hyperbolic with two positive real eigenvalues, and odd otherwise. Since transversally orientable Reeb-Anosov flows only have hyperbolic periodic orbits with both positive eigenvalues, there are only even orbits. Thus the differentials in the chain complex are all zero, and the result follows. Note that Macarini--Paternain prove that result without the assumption of transversal orientability.

It seems highly probable that this argument can at least be extended to all non-degenerate Reeb flows admitting only hyperbolic orbits (which is an a priori strictly larger class than the Reeb-Anosov flows). So, morally, the only free homotopy classes that may not contribute to contact homology are the ones containing elliptic orbits.
It would be very interesting to decide what more general conditions on a Reeb flow $R$ could imply that $[g]\in \cP(R)$ one has $C\bH^{[g]}_{\text{cyl}}(\alpha_1)\neq \{0\}$. 

Work of Colin--Honda--Laudenbach \cite{CHL} imply that \emph{any} area preserving diffeomorphism, in certain fixed isotopy classes, on a surface with one boundary can be realized as the return map to a Birkhoff section (see Definition \ref{def_Birkhoff}) of a Reeb flow. Hence, it is certainly not true that $\cP(R_1)= \cP(R_2)$ for every Reeb flow of a given contact structure.
%

Notice further that we don't need all the differentials of the contact homology chain complex to be zero (as in the Reeb-Anosov case), only that the homology itself is non-zero to deduce information about the free homotopy data. Hence, we ask:
%
\begin{question}\label{que_rmin}
Let $R$ be a Reeb flow of a dynamically convex contact $1$-form $\alpha$. Are there more general conditions on $R$ to ensure that $C\bH^{[g]}_{\text{cyl}}(\alpha)\neq \{0\}$ for any $[g]\in \cP(R)$? 

Are there at least some general conditions on $\alpha$ ensuring the existence of a Reeb flow $R_{\mathrm{min}}$ such that $C\bH^{[g]}_{\text{cyl}}(\alpha)\neq \{0\}$ for any $[g]\in \cP(R_{\mathrm{min}})$?
%
\end{question}

 Proposition \ref{prop_Anosov_contact_implies_Anosov_supporting}, together with Theorem \ref{thm_free_homotopy_data}, gives one direction of Theorem \ref{thm_BMB}, i.e., the fact that any two Reeb-Anosov flows of the same contact structure are isotopically equivalent.
 
I will not go into details on the other direction of Theorem \ref{thm_BMB}, but it consists in showing (using results of Asaoka--Bonatti--Marty \cite{ABM}) that if two Reeb-Anosov flows are isotopically equivalent, with contact structures $\xi_1,\xi_2$, then one can build an open-book decomposition that supports both $\xi_1$ and (a structure isotopic to) $\xi_2$. Celebrated work of Giroux then implies that the $\xi_i$ are isotopic. See \cite[Theorem A.2]{BMB} for details.

\begin{rem}
In \cite{Zung}, Jonathan Zung managed to extend Theorem \ref{thm_BMB} in another direction, by replacing contact structures by \emph{stable Hamiltonian structures}, and using rational symplectic field theory instead of cylindrical contact homology.
\end{rem}

In the next section, I will discuss one type of contact structure, linked to Anosov flows, for which I would hope the above non-vanishing of cylindrical contact homology could hold.

\section{Anosov-supporting bicontact structures}\label{sec_anosov_supporting}

\begin{definition}\label{def_bicontact}
A \emph{bicontact structure} on $M$ is a pair $(\xi^+,\xi^-)$ of (coorientable) transverse contact structures, such that $\xi^+$ is positive and $\xi^-$ is negative\footnote{A coorientable contact structure $\xi$ is called \emph{positive} (resp.~\emph{negative}) if $\alpha \wedge d\alpha >0$ (resp.~$\alpha \wedge d\alpha <0$) for any contact form $\alpha$ defining $\xi$.}.

Given any $C^1$ vector field $X$, if $X\in \xi^+\cap \xi^-$, we say that the bicontact pair \emph{supports} $X$.
\end{definition}

Mistumatsu \cite{Mit95} and Eliashberg--Thurston \cite{ET98} noticed that any transversally orientable Anosov flows is supported by a bicontact structure, and showed that the existence of a supporting bicontact structure forces (and is in fact equivalent to) the flow admitting a \emph{dominated splitting}. 

\begin{definition}
A vector field $X$ is \emph{projectively Anosov} if $PX^t$\footnote{This flow has been called the \emph{linear Poincaré flow} for instance in \cite{ARH}.} the flow induced by $X^t$ on its normal bundle $TM/\langle X\rangle$, admits a dominated splitting. That is 
there exists a splitting $TM\langle X\rangle = E^{s} \oplus E^{u}$ that is invariant by $PX^t$ and such that
\begin{enumerate}[label=(\roman*)]
\item There exists $a,b>0$ such that for all $x\in M$ and $t\geq 0$ we have 
\[
\lVert PX^t|_{(E^{s}/\langle X\rangle)(x)}\rVert \lVert \left(PX^t|_{(E^{u}/\langle X\rangle)(x)}\right)^{-1}\rVert \leq ae^{-bt}.
\]
\item the distributions $E^{cs}$ and $E^{cu}$ on $TM$ induced by taking the preimages of $E^s$ and $E^u$ are coorientable.
\end{enumerate} 
\end{definition}

The coorientability condition is a priori not necessary in the definition, but it has been, as far as I know, a standing assumption in all the literature on projectively Anosov flows. That assumption is necessary to be able to link projectively Anosov flows and bicontact structures:
\begin{proposition}[Mitsumatsu \cite{Mit95}, Eliashberg--Thurston \cite{ET98}]
A vector field is projectively Anosov if and only if it is supported by a bicontact structure.

Moreover, if $X$ is supported by $(\xi^+, \xi^-)$, then $\xi^+, \xi^-, E^{cs}$ and $E^{cu}$ are pairwise transverse, with common intersection along $\langle X\rangle$, and, for all $x\in M$, $\xi^+(x)$, and $\xi^-(x)$ intersects opposite quadrants of $T_xM \smallsetminus \{E^{cs}, E^{cu}\}$. See Figure \ref{fig_quadrants}.
\end{proposition}

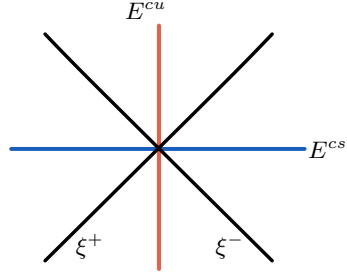
\begin{figure}[h]
	\labellist 
	\small\hair 2pt
	\pinlabel $E^{cs}$ at 230 90 
	\pinlabel $E^{cu}$ at 100 190
	\pinlabel $\xi^-$ at 160 20 
	 \pinlabel $\xi^+$ at 60 20 
	\endlabellist
	\centerline{ \mbox{
			\includegraphics[width=4cm]{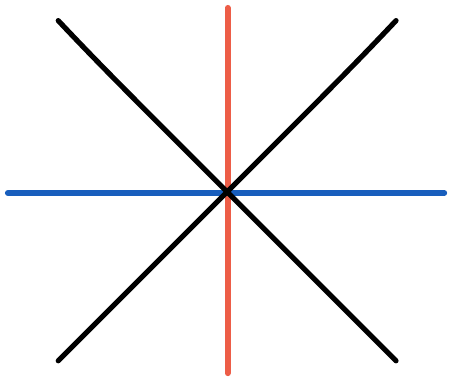}}}
	\caption{The contact structures $\xi^\pm$ in their respective quadrants} 
		\label{fig_quadrants} 
\end{figure}

In fact, one obtains the invariant distributions $E^{cu}$ (resp.~$E^{cs}$) as a limit as $t\to +\infty$ (resp.~$t\to -\infty$) of $(X^t)_\ast \xi^\pm$. Notice that a given projectively Anosov vector field $X$ is always supported by many distinct bicontact structures, for instance, for any choice of $t_1, t_2$, $\left((X^{t_1})_\ast \xi^+,(X^{t_2})_\ast \xi^-\right)$ supports $X$.

Hozoori \cite{Hoz24} proved that the existence of Reeb flows of $\xi^\pm$ that are transverse to both invariant distributions $E^{cu}$ and $E^{cs}$, characterizes Anosov flows amongst projectively Anosov ones. To state his result, we thus introduce another terminology:
\begin{definition}
Given a bicontact structure $(\xi^+, \xi^-)$, a Reeb flow $R^\pm$ for $\xi^\pm$ is said to be \emph{bitransverse} if $R^\pm$ is transverse to \emph{both} $E^{cs}$ and $E^{cu}$.
\end{definition}

 With this definition, Hozoori's result can be stated as follows
\begin{theorem}[Hozoori]\label{thm_Hozoori}
Let $X$ be a projectively Anosov vector field. The following are equivalent
\begin{enumerate}
\item $X$ is Anosov;
\item There exists a bicontact structure $(\xi^+,\xi^-)$ supporting $X$ and such that a Reeb field $R^\pm$ of $\xi^\pm$ is bitransverse.
\item For every bicontact structure $(\xi^+,\xi^-)$ supporting $X$ there exists Reeb fields $R^+$ of $\xi^+$ and $R^-$ of $\xi^-$ that are bitransverse.
\end{enumerate}
\end{theorem}

\begin{definition}\label{def_Anosov_supporting_contact}
If $\xi^\pm$ is a contact structure that is part of a bicontact pair $(\xi^+,\xi^-)$ that supports an Anosov flow $X$, we say that $\xi^\pm$ is an \emph{Anosov-supporting} contact structure.
\end{definition}

Another consequence of Hozoori's work is that a bitransverse Reeb field must be in a specific quadrant of $TM \smallsetminus \{E^{cs}, E^{cu}\}$:
\begin{proposition}[Hozoori \cite{Hoz24b}] \label{prop_bitransverse_implies_adapted}
If $R^\pm$ is a bitransverse Reeb flow of $\xi^\pm$, where $(\xi^+,\xi^-)$ is a bicontact structure supporting $X$, then $R^\pm$ stays in one of the connected component of $TM \smallsetminus \{E^{cs}, E^{cu}\}$ not containing $\xi^\pm$.
\end{proposition}

While not explicitly stated in his article, the above result is a direct consequence of Lemma 3.1 in \cite{Hoz24b}. In particular, it shows that if $R^\pm$ is a bitransverse Reeb flow, then it is \emph{dynamically negative} (for $R^+$) or \emph{dynamically positive} (for $R^-$) in the terminology of \cite{Hoz24}. It is also equivalent to the contact $1$-form $\alpha^\pm$ defining $R^\pm$ to be \emph{adapted} in the terminology of \cite{Hoz24b}.

As Hozoori's Theorem is not stated exactly as above in his work, we provide a quick explanation:
\begin{proof}[Proof of Theorem \ref{thm_Hozoori}]
Theorem 1.10 in \cite{Hoz24} (restated as Corollary 3.3 in \cite{Hoz24b}) says that $X$ is an Anosov flow if and only if there exists a bicontact structure $(\xi^+,\xi^-)$ supporting $X$ and such that a Reeb flow, say $R^+$ of $\xi^+$ is dynamically negative.
The fact that this condition holds for one bicontact pair if and only if it if holds for all pairs follows from a quick argument as in \cite[Lemma 5.3]{Hoz24b}.
Finally, Proposition \ref{prop_bitransverse_implies_adapted} says that the condition of being dynamically positive/negative is equivalent to the bitransverse condition.
\end{proof}

Note that Hozoori, as well as Massoni \cite{Mass23}, gave another characterization of the Anosovity of a projectively Anosov flow in terms of the existence of certain Liouville symplectic forms. Keeping with the spirit of this note, where I am more interested in the dynamics of Reeb flows, I will not use these other characterizations.

A very useful consequence of Hozoori's Theorem is:
\begin{corollary}[Hozoori]\label{cor_Hozoori}
Let $R^+$ be the Reeb flow of a contact $1$-form $\alpha^+$ such that 
\begin{enumerate}[label=(\roman*)]
\item $(\ker \alpha^+, \xi^-)$ is a bicontact structure,
\item $R^+ \in \xi^-$.
\end{enumerate}
Then any $X\in \ker \alpha^+\cap \xi^-$ is Anosov.
\end{corollary}

Among all contact structures, it seems difficult to directly detect which are Anosov contact or which are Anosov-supporting (without actually exhibiting the Anosov Reeb flow, or the Anosov-supporting bicontact pair). There are a few known properties that such structure must satisfy, for instance they must be hypertight (see \cite[Theorem 6.4]{Hoz24}). But satisfying these known properties is a priori very far from being a sufficient condition. 
(See also \cite{Mass23,CLMM} for some of the properties of the associated Liouville symplectic structures.)
 
 As a first step in trying to better understand these contact structures, we will discuss the relationship between the class of Anosov contact structures and that of Anosov-supporting contact structures.
\begin{question}\label{que_Anosov_supporting_and_Anosov}
When is an Anosov-supporting contact structure Anosov?
\end{question}
I believe that Anosov-supporting contact structure should rarely be Anosov (see Conjectures \ref{conj_trichotomy} and \ref{conj_skew_and_regulating} for a precise statement). Indeed, work of Fenley \cite{Fen05} implies the following
 \begin{theorem}\label{thm_bitransverse_Anosov_imply_skew}
 Let $(\xi^+,\xi^-)$ be a bicontact structure supporting an Anosov flow $X$. If $\xi^+$ admits a bitransverse Reeb flow $R^+$ that is Anosov, then $X$ is (positively) skew and isotopically equivalent to $R^+$.
\end{theorem}
\begin{rem}
In fact, for the above result to hold, it is enough to assume that the Reeb-Anosov flow $R^+$ is transverse to just \emph{one} of the stable or unstable foliations of $X$.
\end{rem}

\begin{figure}[h]
	\labellist 
	\small\hair 2pt
	\pinlabel $E^{cs}_X$ at 160 20  
	\pinlabel $E^{cu}_X$ at 120 190
	\pinlabel $X$ at  80 180 %
	 \pinlabel $\xi^+$ at -5 120 
  \pinlabel $R^+$ at 70 5 
	\endlabellist
	\centerline{ \mbox{
			\includegraphics[width=4cm]{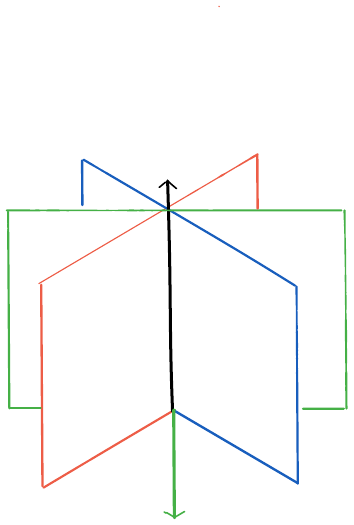}}}
	\caption{A bitransverse Reeb vector field $R^+$} 
		\label{fig_Anosov_xi} 
\end{figure}
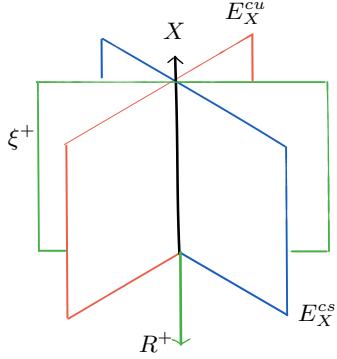

Since positively skew Anosov flows cannot be orbit equivalent (via a map preserving the orientation of $M$, so in particular cannot be isotopically equivalent) to negatively skew ones, an immediate corollary of the above result is:
\begin{corollary}
Let $(\xi^+,\xi^-)$ be a bicontact structure supporting an Anosov flow $X$. If $\xi^+$ admits a bitransverse Reeb flow $R^+$ that is Anosov, then $\xi^-$ cannot admit a bitransverse Anosov Reeb flow.
\end{corollary}

To show the theorem, we will use the following result of Fenley about (pseudo)-Anosov flows transverse to foliations.
\begin{theorem}[Fenley \cite{Fen05}]\label{thm_fenley_transverse_to_R-covered}
Let $Y$ be a skew Anosov flow transverse to a foliation $\cF$. Then $\cF$ is $\bR$-covered. Moreover, $\cF$ is topologically equivalent to a \emph{blow-up} of the stable or unstable foliation of $Y$.
\end{theorem}

Since this statement is not directly written in \cite{Fen05}, we provide the argument:
\begin{proof}
Theorem C of \cite{Fen05} gives that $\cF$ must be $\bR$-covered, and that either $Y$ is regulating (see Definition \ref{def_regulating}), or $\cF$ is topologically equivalent to a blow-up of the stable or unstable foliation of $Y$. The next lemma shows that only suspension Anosov flows can be regulating to a foliation, and the assumption says that $Y$ is not a suspension.
\end{proof}

\begin{lemma}\label{lem_regulating_implies_no_pf}
If $Y$ is a pseudo-Anosov flow regulating to a foliation $\cF$, then either $Y$ is a suspension of an Anosov diffeomorphism or has some singular orbits.
\end{lemma} 
\begin{rem}
While we do not need it here, one can be more precise, and show that a pseudo-Anosov flow that is regulating to a foliation has no perfect fits.
\end{rem}

\begin{proof}
Suppose that $Y$ is an Anosov flow (i.e., does not have any singular orbit). If $Y$ is not a suspension, then there exists periodic orbits $\alpha_1,\alpha_2$ that are freely homotopic to the inverse of each other (see, e.g., \cite[Theorem 2.15]{BBGRH}). This is easily seen to be incompatible with being regulating to a foliation.
\end{proof}

Using Fenley's result, we can quickly deduce Theorem \ref{thm_bitransverse_Anosov_imply_skew}

\begin{proof}[Proof of Theorem \ref{thm_bitransverse_Anosov_imply_skew}]
By assumption, $R^+$ is Anosov and transverse to, say, $\cF^s_Y$, the stable foliation of $Y$. By Theorem \ref{thm_fenley_transverse_to_R-covered}, we deduce that $\cF^s_Y$ is $\bR$-covered and is topologically equivalent to, a priori a blow-up of, $\cF_{R^+}^s$, the stable foliation of $R^+$. But one easily sees that this implies that $Y$ and $R^+$ are orbit equivalent. This can be deduced directly from the arguments in \cite{Fen05}, but here, we will use instead Theorem \ref{thm_free_homotopy_data}: An element $g\in \pi_1(M)$ is in $\cP(Y)$ if and only if it is freely homotopic to a closed curve in a cylindrical leaf of $\cF^s_Y$. Since $\cF^s_Y$ is a blow-up of $\cF_{R^+}^s$, $g$ represents a cylindrical leaf of $\cF^s_Y$ if and only if it represents a cylindrical leaf of $\cF_{R^+}^s$. Hence $\cP(Y) = \cP(R^+)$, and as they are both $\bR$-covered Anosov flows, Theorem \ref{thm_free_homotopy_data} imply that they are isotopically equivalent.
\end{proof}

Theorem \ref{thm_bitransverse_Anosov_imply_skew} gives a reason to believe that Anosov-supporting contact structure that are also Anosov should be in some sense rare (or rigid). Now we consider their existence:

\begin{proposition}\label{prop_Anosov_contact_implies_Anosov_supporting}
Let $\beta$ be an Anosov contact structure, then $\beta$ is Anosov-supporting. 
More precisely, if $\beta$ is a positive contact structure, $X$ is an Anosov Reeb flow of $\beta$ and $(\xi^+,\xi^-)$ is a bicontact structure supporting $X$, then $(\beta,\xi^-)$ is a bicontact structure supporting an Anosov flow $Y$ that is isotopically equivalent to $X$.

In particular, $Y$ is the \emph{unique}, up to orbit equivalence, Anosov flow supported by $\beta$.
\end{proposition}

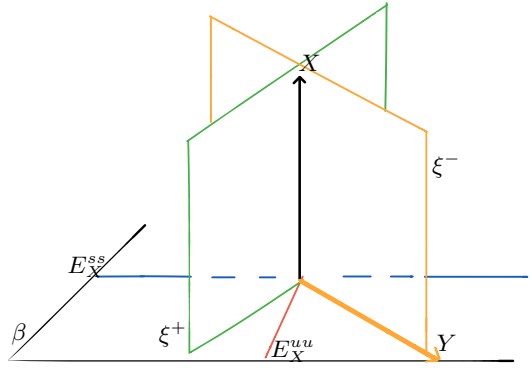
\begin{figure}[h]
	\labellist 
	\small\hair 2pt
	\pinlabel $E^{ss}_X$ at 50 60 
	\pinlabel $E^{uu}_X$ at 170 10
	\pinlabel $\xi^-$ at 260 120 
	 \pinlabel $\xi^+$ at 100 20 
\pinlabel $\beta$ at 10 20 
  \pinlabel $X$ at 180 180 
   \pinlabel $Y$ at 260 15 
	\endlabellist
	\centerline{ \mbox{
			\includegraphics[width=7cm]{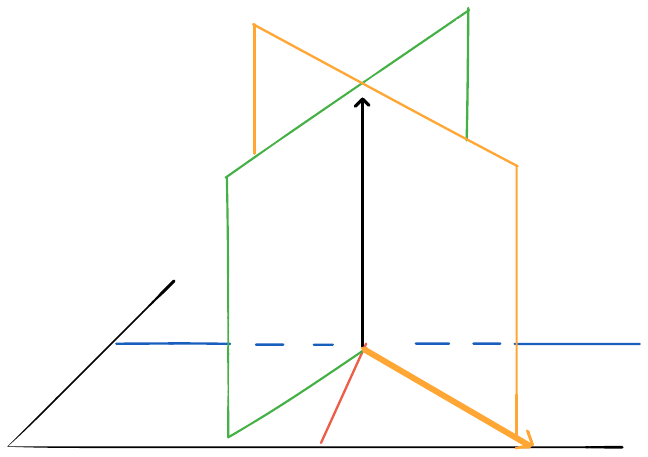}}}
	\caption{The vector fields $X$, $Y$ and the three contact structures} 
		\label{fig_contact_bicontact} 
\end{figure}

 Our proof of this proposition again uses Fenley's work, as well as Hozoori's theorem.
\begin{proof}
First $\beta$ and $\xi^-$ are of distinct signs, coorientable, and transverse (since $X$ is in  $\xi^-$ but never in $\beta$), so they define a bicontact structure. Then, as $X$ is a Reeb flow of $\beta$ and $X\in \xi^-$, Corollary \ref{cor_Hozoori} implies that $Y \in \beta\cap\xi^-$ is Anosov.

Now, since $X$ is Reeb Anosov, $\beta = E^{ss}_X \oplus E^{uu}_X$, and since $(\xi^+,\xi^-)$ supports $X$, $\xi^-\cap \beta$ is transverse to both $E^{ss}_X$ and $E^{uu}_X$. In particular, as $Y\in \xi^-\cap \beta$, it is transverse to $\bR\cdot X \oplus  E^{ss}_X$ and to $\bR\cdot X \oplus  E^{uu}_X$. See Figure \ref{fig_contact_bicontact}. Since $X$ is Reeb-Anosov, its weak stable foliation is $\bR$-covered, and so $Y$ is an Anosov flow transverse to an $\bR$-covered foliation.
By \cite[Main Theorem]{Fen05}, if $Y$ was not $\bR$-covered, then it would have to be regulating to $\cF_X^s$. By Lemma \ref{lem_regulating_implies_no_pf}, this would imply that $Y$ must be a suspension Anosov flow. But then $M$ is a solvmanifold, which is incompatible with admitting a Reeb-Anosov flow (by Plante \cite{Plante81}, any Anosov flow on a solvmanifold is a suspension of an Anosov diffeomorphism).

Thus, we obtained that $Y$ is $\bR$-covered, and (since $M$ is not a solvmanifold) must be skew. Since $Y$ is transverse to $\cF^s_X$, and is skew, we can now use Theorem \ref{thm_fenley_transverse_to_R-covered} to deduce that $\cF^s_X$ must be a blow-up of the stable foliation of $Y$. Then, as in the end of the proof of Theorem \ref{thm_bitransverse_Anosov_imply_skew}, we deduce from Theorem \ref{thm_free_homotopy_data} that $X$ and $Y$ are isotopically equivalent. 
\end{proof} 

 Jonathan Bowden and Thomas Massoni gave another argument to show that if $\beta$ is an Anosov contact structure, then it is Anosov-supporting:
 \begin{proposition}[Bowden--Massoni, work in preparation]
 If $\beta$ is a positive contact structure, $X$ is an Anosov Reeb flow of $\beta$ and $(\xi^+,\xi^-)$ is a bicontact structure supporting $X$, then $\beta$ and $\xi^+$ are isotopic.
 \end{proposition}


Given Proposition \ref{prop_Reeb_flows_same_free_hom_data}, an interesting consequence of Proposition \ref{prop_Anosov_contact_implies_Anosov_supporting}, is that the free homotopy data of a Reeb-Anosov flow of an Anosov contact structure is also equal to that of a \emph{tangential} Anosov flow:
\begin{corollary}
Let $\beta$ be an Anosov contact structure and $Y$ any Anosov flow supported by a bicontact pair $(\beta,\xi^-)$. Then for any nondegenerate Reeb flow $R$ of $\beta$, $\cP(R) \supset \cP(Y)$ with equality at least when $R$ is Anosov.
\end{corollary}

Theorem \ref{thm_bitransverse_Anosov_imply_skew} together with Proposition \ref{prop_Anosov_contact_implies_Anosov_supporting}, suggest the following conjectural trichotomy for bicontact pairs:

\begin{conjecture}\label{conj_trichotomy}
Let $(\xi^+,\xi^-)$ be a bicontact structure supporting an Anosov flow $X$. Then, we have one of the following exclusive possibilities:
\begin{enumerate}
\item $\xi^+$ is Anosov contact, it admits a Reeb Anosov flow $R^+$ contained in $\xi^-$ and $X$ is a positively skew Anosov flow isotopically equivalent to $R^+$;
\item $\xi^-$ is Anosov contact, it admits a Reeb Anosov flow $R^-$ contained in $\xi^+$ and $X$ is a negatively skew Anosov flow isotopically equivalent to $R^-$;
\item Neither $\xi^+$ nor $\xi^-$ is Anosov contact and either $X$ is a suspension Anosov flow or is not $\bR$-covered.
\end{enumerate}
\end{conjecture} 

Recall that Marty proved that an Anosov flow is positively skew if and only if it is istopically equivalent to a Reeb Anosov flow (Theorem \ref{thm_martyequivalence}), the above conjecture gives a more precise version of where such a Reeb-Anosov flow should be found.
In order to prove this conjecture, one needs to be able to remove the condition of bitransversality in Theorem \ref{thm_bitransverse_Anosov_imply_skew}.

The above conjecture so far leaves unsaid what happens to the Reeb flows of, say, $\xi^-$ when $\xi^+$ is Anosov contact. Based on very limited data\footnote{Prettiness of the purported result being one important factor.}, we further conjecture:

\begin{conjecture}\label{conj_skew_and_regulating}
Let $(\xi^+,\xi^-)$ be a bicontact structure supporting an Anosov flow $X$. The flow $X$ is positively skew if and only if 
\begin{itemize}
\item There exists a bitransverse Reeb flow $R^+$ of $\xi^+$ that is Anosov, and
\item There exists a bitransverse Reeb flow $R^-$ of $\xi^-$ that is regulating for both stable and unstable foliations of $X$.
\end{itemize}
\end{conjecture}

If this conjecture holds, then in particular, we get that $\cP(R^+) = \cP(X)$ (by Theorem \ref{thm_bitransverse_Anosov_imply_skew}) and $\cP(R^-) \cap \cP(X) = \emptyset$ (because an element of $\pi_1(M)$ that represents any periodic orbit of a regulating flow will act as a translation on the leaf space, so cannot represent a periodic orbit of $X$).

Note that Conjecture \ref{conj_skew_and_regulating} is at least true in the case of the geodesic flow. Using Salmoiraghi's bicontact surgeries (\cite{Sal23,Sal24}), one can also show by hand that this holds for instance for the Handel--Thurston Anosov flow. More generally, in the setting of the conjecture one can projects the orbits of a bitransverse Reeb flow $R^-$ to the orbit space of $X$ to see that it appears that the flow should be regulating. Alternatively, one may look at the \emph{slithering} picture of skew Anosov flows (see \cite{Thurston:3MFC}), and try to understand what orbits of $\wt R^-$ do.

The case of non $\bR$-covered Anosov flow and its relation to the Reeb flows of bicontact structure is so far much more mysterious to me. For instance:
\begin{question}
Let $X$ be a non $\bR$-covered Anosov flow supported by $(\xi^+,\xi^-)$ and let $R^\pm$ be a bitransverse Reeb flows of $\xi^\pm$. Can one have $\cP(R^\pm) = \cP(X)$ or $\cP(R^\pm) \cap \cP(X) = \emptyset$?
\end{question}
My guess is that, in the non-$\bR$ covered case, the intersection of $\cP(R^\pm)$ and $\cP(X)$ should always be proper and, at least generically, non-empty. 
To be a little more specific, I need to recall one more feature of the orbit space of pseudo-Anosov flows: A \emph{lozenge} $L$ is an open, trivially foliated region in $\cO$ whose boundary consists of four half-leaves of $\wt \cF^s(x)$, $\wt \cF^u(x)$, $\wt \cF^s(y)$ and $\wt \cF^u(y)$ such that $\wt \cF^s(x)\cap \wt \cF^u(y) = \emptyset$ and $\wt \cF^s(y)\cap \wt \cF^u(x) = \emptyset$. The points $x$ and $y$ are called the \emph{corners} of $L$ and the four boundary segments are called the \emph{sides}. See Figure \ref{fig_lozenge}, and, e.g., \cite{Potrie_notes} for more details.
When an Anosov flow is transversally orientable, and we fixed a choice of transverse orientations for both foliations, then one can distinguish two types of lozenges: A lozenge is \emph{positive} if the two sides of one corner are both positive half-leaves (and equivalently, the two sides of the other corner are both negative half leaves). Otherwise it is called \emph{negative}.

\begin{figure}[h]
	\labellist 
	\small\hair 2pt
	\pinlabel $x$ at 45 85 
	\pinlabel $y$ at 190 85
	\pinlabel $L$ at 110 100 
	\endlabellist
	\centerline{ \mbox{
			\includegraphics[width=4cm]{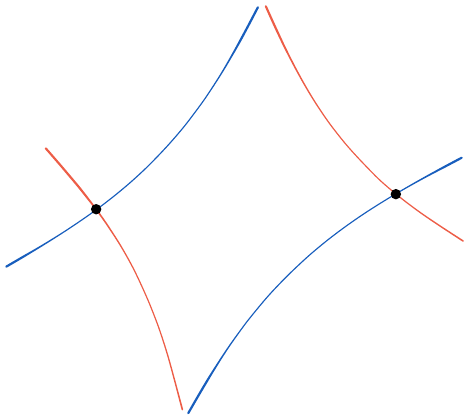}}}
	\caption{A lozenge and its corners} 
		\label{fig_lozenge} 
\end{figure}

I hope the following is true:
\begin{conjecture}
Let $X$ be an Anosov flow supported by $(\xi^+,\xi^-)$, then there exists bitransverse Reeb flows $R^+$ and $R^-$ such that, for any $[g]\in \cP(X)$, we have $[g]\in \cP(R^+)$ (resp.~$[g]\in \cP(R^-)$) if and only if $g$ fixes a positive (resp.~negative) lozenge. 
\end{conjecture}

One direction of the above conjecture could possibly be proved in the following way: If $g$ leaves invariant a lozenge, then this lozenge corresponds to the projection of a \emph{Birkhoff annulus}, i.e., an annulus transverse to the Anosov flow $X$ in its interior and bounded by two orbits of $X$. (This result is originally due to Barbot in \cite{Barbot_MPOT}, see also \cite{ABM} for an improved version of that result.) It seems likely that one can realize a Birkhoff annulus of a, say, positive lozenge, by flowing one of the corner orbits by an (appropriately chosen) Reeb flow of the \emph{negative} contact structure $\xi^-$. Then one can choose a Reeb flow of $\xi^+$ to have a periodic orbit in the middle of that annulus. Note that these type of well chosen Reeb flows for Birkhoff annuli is similar to the local coordinates chosen in the surgeries of Salmoiraghi \cite{Sal23,Sal24}, or the Foulon--Hasselblatt surgery \cite{FH13}.
Being able to choose one pair of Reeb flows that would realize these nice Birkhoff annuli for every lozenge at the same time however would require much more work.

 \section{Birkhoff sections and pseudo-Anosov models}\label{sec_pAmodels}

In this final section, we discuss the existence of pseudo-Anosov models (see Definition \ref{def_pAmodel}) for more general $3$-dimensional flows.

Morally, a way of building a pseudo-Anosov model to a $3$-dimensional flow $X^t$ is by using the return map to a Birkhoff section (when $X$ admits such a section), provided this return map is isotopic to a pseudo-Anosov representative. We now try to make this strategy more precise.

\begin{definition}\label{def_Birkhoff}
A \emph{Birkhoff section} for a vector field $X$ on $M$ is an immersion $\iota \colon S \to M$ defined on a compact surface $S$ satisfying:
\begin{enumerate}[label =(\roman*)]
\item If $\partial S\neq \emptyset$, then $\iota(\partial S)$ is a union of periodic orbits of $X$;
\item $\iota^{-1}\left(\iota(\partial S)\right) = \partial S$ and $\iota$ defines an embedding $S\smallsetminus \partial S \to M\smallsetminus \iota(\partial S)$ that is transverse to the flow;
\item for every $x\in M$, there exists $t_0<0<t_1$ such that $X^{t_i}(x)\in \iota (S)$, for $i=0,1$.
\end{enumerate}
If $\iota$ is an embedding on $S$, we say that the Birkhoff section is an \emph{embedded} Birkhoff section.
\end{definition}

Note that given a Birkhoff section $\iota\colon S \to M$ for $X$, one obtains homeomorphisms on $\iota(S)$ and on $S$, by considering the first return map of $X$ on the interior of $\iota(S)$. We call $f_{X;\iota(S)}$ this return map seen as a map on $\iota(S)$, and $f_{X;S}$ its counterpart on $S$.

Topologically, the manifold $M\smallsetminus \{\iota(\partial S)\}$ is the mapping torus of $f_{X;S}$ when restricted to the interior of $S$. Moreover, by construction, $X$ restricted to $M\smallsetminus \{\iota(\partial S)\}$ is orbit equivalent to the suspension of $f_{X;S}$, and this orbit equivalence can be extended to $M$ by adding the periodic orbits corresponding to $\iota(\partial S)$.

Not all three manifold flows admit Birkhoff sections. For instance, a pseudo-Anosov flows admits a Birkhoff section if and only if it is transitive \cite{Fri83}. For Reeb flows, however, a recent result of Colin--Dehornoy--Hryniewicz--Rechtman \cite{CDHR24} (see also Contreras--Mazzucchelli \cite{CM22}) shows that admitting a Birkhoff section is a generic property:
\begin{theorem}[\cite{CDHR24}, Corollary 1.2]
The set of contact $1$-forms whose Reeb flows admit Birkhoff sections is open and dense in the $C^{\infty}$-topology.
\end{theorem}

If a flow $X$ admits a Birkhoff section, one can then use the Nielsen--Thurston classification of surface diffeomorphisms to build a model flow by choosing a different representative in the isotopy class of the return map $f_{X;S}$. 
One minor technical issue appears when the Birkhoff section is not embedded, as $f_{X;S}$ may then permute the boundary components of $S$, so for simplicity, we will now restrict to flows admitting embedded Birkhoff sections\footnote{In a very recent work \cite{CHR25}, Colin, Hryniewicz and Rechtman show that Reeb flows admitting embedded Birkhoff sections is also $C^{\infty}$-generic, so our assumption is generically verified.}. 

Let $\Pmod(S)$ be the pure mapping class group of $S$, i.e., the space of diffeomorphisms of $S$ preserving each boundary components, up to isotopy. We write $\{f\}\in \Pmod(S)$ for the isotopy class of a diffeomorphism $f$.

\begin{proposition}[Handel]
Let $X$ be a flow on $M$ admitting an embedded Birkhoff section $\iota \colon S \to M$. Suppose that $\{f_{X;S}\}\in \Pmod(S)$ is pseudo-Anosov. Call $f$ a pseudo-Anosov representative in the isotopy class of $f_{X;S}$.

Then there exists a, possibly $1$-pronged, pseudo-Anosov flow $\bar X$ on $M$, such that $\iota \colon S \to M$ is a Birkhoff section of $\bar X$ with return map $f_{\bar X; S} = f$.

Moreover, $\cP(\bar X) \subset \cP(X)$ and there exists a compact, $X$-invariant, subset $K\subset M$ and a continuous and surjective map $h \colon K \to M$ that takes flowlines of $X|_K$ to flowlines of $\bar X$. 
\end{proposition}

If $\bar X$ happens to be a true pseudo-Anosov flow, i.e., does not admit any $1$-prong singularities, then $\bar X$ is a pseudo-Anosov model for $X$ in the sense of Definition \ref{def_pAmodel}.

\begin{proof}[Sketch of proof]
As discussed above, one can take the suspension of $f$ to obtain a flow on $M\smallsetminus \{\iota(\partial S)\}$, and then adding $\iota(\partial S)$ as periodic orbits to obtain a flow $\bar X$. By construction, $\iota \colon S\to M$ is a Birkhoff section for $\bar X$, with $f_{\bar X;S}=f$ being a pseudo-Anosov diffeomorphism. Then $\bar X$ is a pseudo-Anosov flow on $M\smallsetminus \{\iota(\partial S)\}$, but the orbits given by $\iota(\partial S)$ may be singular, and in particular, they may be $1$-prong. Thus $\bar X$ is a possibly $1$-pronged pseudo-Anosov flow.

Now, Thurston proved that pseudo-Anosov diffeomorphisms minimizes the number of periodic points in its isotopy class (see, e.g., \cite{FLP79,FM_primer}), which implies that $\cP(\bar X)\subset \cP(X)$.

Moreover, Handel \cite[Theorem 2 and Remark 2.4]{Han85} proved that there exists a closed set $C$ in the interior of $S$ and a surjective map $h_S\colon C \to \mathring S$, homotopic to inclusion, such that $f\circ h_S = h_S \circ f_{X;S}|_C$. Flowing $C$ by $X$ and adding $\iota(\partial S)$ gives the set $K$, and similarly one can build the map $h$ from $h_S$.
\end{proof}

Note that when the manifold $M$ is hyperbolic, then the isotopy class of any return maps $f_{X;S}$ will have to admit a pseudo-Anosov representative. So we deduce:
\begin{corollary}
If $M$ is a hyperbolic $3$-manifold and $X$ is a flow admitting an embedded Birkhoff section, then $X$ admits a possibly $1$-pronged pseudo-Anosov model.
\end{corollary}

I stated several questions regarding existence and uniqueness of pseudo-Anosov models in the introduction. I will end this note with some further questions.

First, since $1$-pronged pseudo-Anosov flows may behave very differently from true pseudo-Anosov flows, it would be nice to be able to know when one can obtain a true pseudo-Anosov model, as opposed to a $1$-pronged one. The next question is then a more general version of Question \ref{que_existencepAmodel} of the introduction.
\begin{question}
Are there (easily verifiable) conditions ensuring that a flow has a true (not $1$-pronged) pseudo-Anosov model? Can a given flow $X$ admit both a $1$-pronged and a true pseudo-Anosov model?
\end{question}

Some hyperbolic $3$-manifolds do not admit true pseudo-Anosov flows (see, e.g., \cite{CD03,Fen07,RSS03}), at least for those manifolds, one would have to deal with $1$-pronged pseudo-Anosov flows, which are still very much mysterious. 
\begin{question}
Can one extend Theorem \ref{thm_free_homotopy_data} to $1$-pronged pseudo-Anosov flows on hyperbolic $3$-manifold?
\end{question}

\begin{question}
Suppose $X$ admits a pseudo-Anosov model $\bar X$. Can one describe which subsets of $\cP(X)$ can be realized as $\cP(\bar X)$? Does it help to assume that $X$ is a Reeb flow?
\end{question}

If a Reeb flow $R_{\mathrm{min}}$ as in Question \ref{que_rmin} exists, it would be particularly interesting if one had $\cP(R_{\mathrm{min}})= \cP(\bar R_{\mathrm{min}})$ as one could then use Theorem \ref{thm_free_homotopy_data} to obtain the uniqueness of the pseudo-Anosov model.

A general theme of study that we only brushed upon above when mentioning some of Fenley's results, is the relationships between pseudo-Anosov flows and transverse, or \emph{almost-transverse}\footnote{In the sense of Mosher \cite{Mos91}, i.e., one can change the pseudo-Anosov flow near some of its singular orbits by an operation called a \emph{dynamical blow-up} to obtain a transverse flow.} taut foliations. It seems that this relationship may manifest here:

\begin{question}
Let $X$ be an Anosov flow supported by the bicontact structure $(\xi+,\xi^-)$.
Suppose that a bitransverse reeb flow $R^\pm$ admits a pseudo-Anosov representative $\bar R$. Is $\bar R$ transverse or almost-transverse to the weak (un)stable foliation of $X$?
\end{question}

Finally, as in \cite{CHL}, which gives conditions for a surface diffeomorphism to be the return map to an embedded Birkhoff section of a Reeb flow, it would be interesting to detect which pseudo-Anosov (or $1$-pronged pseudo-Anosov) flows can be a model for a Reeb flow in the sense of Definition \ref{def_pAmodel}. A natural candidate for such a condition comes from the work of Jonathan Zung:
\begin{question}
Let $\phi$ be a pseudo-Anosov flow, is it the pseudo-Anosov model of a Reeb flow if and only if 
the associated stable Hamiltonian structure $(\omega_\phi,\lambda_\phi)$ of \cite[Construction 1.2]{Zung} is contact?
\end{question}

\bibliographystyle{amsalpha}
\bibliography{Heidelberg_refs}
 
\end{document}